\documentclass[11pt]{article}
\usepackage{latexsym,amssymb,amsmath,amsfonts,amsthm}
\usepackage{graphics,graphicx,mathrsfs,subfigure}
\usepackage{color,xspace}
\newcommand{\R}{{\mat R}}

\newcommand{\C}{{\mat C}}
\newcommand{\Sp}{{\mat S}}

\newcommand{\be}{\begin{eqnarray}}
\newcommand{\ben}{\begin{eqnarray*}}
\newcommand{\en}{\end{eqnarray}}
\newcommand{\enn}{\end{eqnarray*}}

\newcommand{\mat}{\mathbb}

\newtheorem{theorem}{Theorem}[section]
\newtheorem{lemma}[theorem]{Lemma}

\newtheorem{remark}[theorem]{Remark}

\definecolor{rot}{rgb}{1,0,0}
\definecolor{hw}{rgb}{0,0,1}

\begin{document}
\renewcommand{\theequation}{\arabic{section}.\arabic{equation}}
\title{\bf
 Increasing stability in the n-dimensional inverse source problem with multi-frequencies
}
\author{ Suliang Si\thanks{School of Mathematics and Statistics, Shandong University of Technology,
Zibo, 255000, China ({\tt sisuliang@amss.ac.cn})}}
\date{}



\maketitle

\begin{abstract}
In this paper, we show for the first time the increasing stability of the inverse source problem for the n-dimensional Helmholtz equation at multiple wave numbers, which is different from the two-or three-dimensional Helmholtz equation. In addition, we develop a new, unified approach to study increasing stability in any dimension. The method is based on the Fourier transform and explicit bounds for analytic continuation. 
\end{abstract}

%


\section{Introduction}
\noindent In this paper, we deal with an inverse time-harmonic source  problem for the Helmholtz equation. Our aim is to determine the source function $f(x)$ appearing in the following  equation defined in $\R^n$:
\begin{equation}\label{u}
     \Delta u(x)+k^2u(x)=f(x), \ \,x\in\R^n. \\
\end{equation}
Here $k>0$ is the wave number, $u$ is the radiated wave field and $f\in L^{2}(\R^n)$ is a compactly supported function. We assume that the supp $(f)$ is contained in the ball $B_R$ defined by 
$$B_R:=\{x\in\R^n| \ |x|<R\}$$
for some $R>0$. The following Sommerfeld radiation condition is required to ensure the uniqueness of the wave field $u$:
\begin{equation}
r^{\frac{n-1}{2}}(\partial_r u-iku)=0, \ \,r:=|x|\rightarrow +\infty,
\end{equation}
uniformly in all directions $\hat{x}=x/|x|$.

For a given function $g$ on $\partial B_R$. Consider the following boundary value problem outside $B_R$:
\begin{equation}\label{u2}
\begin{cases}
     \Delta u(x)+k^2u(x)=0, \ \,x\in\R^n\setminus \overline{B_R}, \\
     u=g,  \  \  x\in\partial B_R,\\
    r^{\frac{n-1}{2}}(\partial_r u-iku)=0, \ \,r:=|x|\rightarrow +\infty.  \\ 
\end{cases}
\end{equation}
We assume that $g\in H^{1/2}(\partial B_R)$. It can be shown as in e.g. \cite{CK} that the problem (\ref{u2}) has a unique solution  $u\in H^1_{loc}(\R^n\setminus B_R)$. 
By the trace theorem, we have 
\[\|\frac{\partial u}{\partial\nu}\|_{H^{-1/2}(\partial B_R)}\leq C_1\|u\|_{H^{1}(B_{2R}\setminus\overline{B_R})}\leq C_2\|g\|_{H^{1/2}(\partial B_R)}\]
 for some positive constant $C_1$ and $C_2$. Thus
we may introduce the Dirichlet-to-Neumann (DtN) operator $\mathcal{B}: H^{1/2}(\partial B_R)\rightarrow H^{-1/2}(\partial B_R)$ given by $\mathcal{B}g=\frac{\partial u}{\partial\nu}$ where $u$ is a solution of the problem (\ref{u2}) with the Dirichlet data $u=g$ on $\partial B_R$. Using the DtN operator, we reformulate the Sommerfeld radiation condition into a transparent
boundary condition
$\frac{\partial u}{\partial\nu}=\mathcal{B}g$ on $\partial B_R$,
where $\nu$ is the unit outer normal on $\partial B_R$. Hence one can also obtain the Neumann data on $\partial B_R$ once the
Dirichlet date is available on $\partial B_R$.

Now we are in the position to discuss our inverse source problem:

\textit{\textbf{IP}}. Let $f$ be a complex function with a compact support contained in $B_R$. The inverse problem is to
determine $f$ by using the boundary observation data $u(x,k)|_{\partial B_R}$ with an interval of frequencies $k\in(0,K)$
where $K>1$ is a positive constant.

The study of inverse coefficients and source  problems for partial differential equations is one of the most rapidly growing mathematical research area in the recent years,  and inverse source problems attracted recently much attention. Inverse source problems have indeed many applications in many fields, like for example in antenna synthesis \cite{[1]}, biomedical and medical imaging \cite{[2]} and  tomography \cite{[3],[4]}.

Motivated by these significant applications, the inverse source problems, as an important research subject
in inverse scattering theory, have continuously attracted much attention by many researchers \cite{[BLRX],[BT20],[5],BC,[7],[15],[16],[8],[20],LZZ}. Consequently, a great deal of mathematical and numerical results are available, especially for the acoustic waves or the Helmholtz equations. We state in brief some of the existing results that are relevant to the problem under investigation in this paper. It is well known  that there is an obstruction to uniqueness for inverse source problems for Helmholtz equations with a single frequency data. For the convenience of the reader, we refer to \cite[Chapter 4]{[6]} and \cite{BC}. However,  by considering multi-frequency measurements, the uniqueness can be proved. For this, one can see for example the recent works \cite{[5], [12]} in which uniqueness and stability results have been proved for the recovery of the source term from  knowledge of multi-frequency boundary measurements. In \cite{[7]}, the authors 
treated an interior inverse source problem for the Helmholtz equation from boundary Cauchy data for multiple wave numbers and they showed an increasing stability result with growing $K$ for the problem under consideration. 
Interested readers can also see  \cite{[12]} that claims a uniqueness result and a numerical algorithm for recovering the location and the shape of a supported acoustic source function from boundary measurements at many frequencies.  See also \cite{[BLRX], [8],[11],[15],[16], [20]} and the references therein.
As for increasing stability results proved for coefficients inverse problems, we can refer for example to \cite{[Horst]} and \cite{[U]},  in which  inverse problems of recovering an electric potential appearing in a Schr\"odinger equation have been studied (see also the references therein). In \cite{[BT20]}, the increasing stability for the one-dimensional inverse medium  problem of recovering the refractive index is investigated. However, there are few works on the inverse source problems for the n-dimensional Helmholtz equation and the available results are mainly focused on the two-or three-dimensional Helmholtz equation \cite{[12], [8], [20]}. The stability issue is wide open to be investigated for the n-dimensional Helmholtz equation. The existing results are mainly based on analysis of the fundamental solution for the two- or three- dimensional Helmholtz equation. In this work, we use the Fourier transform and develop a unified approach to study inverse source problems.

The paper is organized as follows. In section \ref{Man}, we briefly show the main result. Section \ref{Inc} is developed to stability analysis of the inverse source problem by using multifrequency data.

\section{Main result}\label{Man}
Let $d$ be an integer and $d>0$, denote a complex-valued functional space:
\[\mathcal{C}_{M,d}=\{f\in H^{2d}(\R^n)| \ ||f||_{H^{2d}(B_R)}\leq M, \ supp(f)\subset B_R, \ f:B_R\rightarrow \C\},\]
where $M>0$ is a constant. For any $v\in H^{1/2}(\partial B_R)$, we set 
\[||v(x,k)||_{\partial B_R}=\int_{\partial B_R}(|\mathcal{B} v(x,k)|^2+k^2|v(x,k)|^2)ds(x).\]

Now we show the main increasing stability result. 
\begin{theorem}\label{1}
Let $f\in\mathcal{C}_{M,d}$. Then there exists a constant $C>0$ depending on $n$, $d$ and $R$ such that 
\begin{equation}\label{f}
 ||f||^2_{L^2{(\R^n)}}\leq C\Big(\epsilon^2+\frac{ M^2}{(K^{\frac{2}{3}}|\ln \epsilon|^{\frac{1}{4}})^{4d-n}}\Big)
\end{equation}
where $K>1$, $4d>n$ and 
\[\epsilon=\big(\int_0^Kk^{n-1}||u(x,k)||_{\partial B_R}dk\big)^{\frac{1}{2}}.\]
\end{theorem}

\begin{remark}
There are two parts in the stability estimates (\ref{f}): the first part is the data discrepancy and the second part comes from the high frequency tail of the function. It is clear to see that the stability increases as $K$ increases, i.e., the problem is more stable as more frequencies data are used.
\end{remark}

\begin{remark}
The idea was firstly proposed in \cite{[7]} by separating the stability into the data discrepancy and high frequency tail where the latter was estimated by the unique continuation for the three-dimensional inverse source scattering problem. Our stability result in this work is consistent with the one in \cite{BLZ,[20]} for both the two and three-dimensional inverse scattering problems.
\end{remark}

\section{Increasing stability}\label{Inc}
Let $\xi\in\R^n$ with $|\xi|=k$.
Multiplying $e^{-i\xi\cdot x}$ on both sides of (\ref{u}) and integrating over $B_R$, we obtain 
\begin{equation}
\nonumber
\int_{B_R}f(x)e^{-i\xi\cdot x}dx=\int_{\partial B_R}e^{-i\xi\cdot x}(\partial_\nu u+i\xi\cdot\nu u)ds(x), \, \  |\xi|=k\in(0,+\infty).
\end{equation}
Since supp $(f)$ is contained in the ball $B_R$, we have 
\begin{equation}
\nonumber
\int_{\R^n}f(x)e^{-i\xi\cdot x}dx=\int_{\partial B_R}e^{-i\xi\cdot x}(\partial_\nu u+i\xi\cdot\nu u)ds(x), \ |\xi|=k\in(0,+\infty)
\end{equation}
which gives 
\begin{equation}\label{Fof}
\big|\widehat{f}(\xi)\big|^2\leq C\int_{\partial B_R}(|\partial_\nu u|^2+k^2|u|^2)ds(x), \ |\xi|=k\in(0,+\infty),
\end{equation} 
where $C>0$ is a constant and depends on $n$ and $R$.
Using the spherical polar coordinates 
\[\xi=k\hat{\xi}=k(\cos\varphi_1,\sin\varphi_1\cos\varphi_2,\cdot\cdot\cdot,\sin\varphi_1\cdot\cdot\cdot \sin\varphi_{n-2}\sin\varphi_{n-1}),\]
we obtain that
\begin{equation}
\nonumber
\begin{split}
\int_{\R^n}\big|\widehat{f}(\xi)\big|^2d\xi=\int_0^{2\pi}d\varphi_{n-1}\int_0^{\pi}\sin\varphi_{n-2}d\varphi_{n-2}\cdot\cdot\cdot\int_0^{\pi}\sin\varphi_1^{n-2}d\varphi_{1}\int_0^{+\infty}k^{n-1}
\big|\widehat{f}(\xi)\big|^2dk.
\end{split}
\end{equation}
It follows from the Plancherel theorem that 
\begin{equation}
\nonumber
\begin{split}
(2\pi)^n||f||^2_{L^2(\R^n)}=||\widehat{f}||^2_{L^2(\R^n)}=\int_{\R^n}|\widehat{f}(\xi)|^2d\xi=\int_{|\xi|\leq s}|\widehat{f}(\xi)|^2d\xi+\int_{|\xi|>s}|\widehat{f}(\xi)|^2d\xi.
\end{split}
\end{equation}
Denote
\begin{equation}\label{I1s}
\begin{split}
I_1(s)&=\int_{|\xi|\leq s}|\widehat{f}(\xi)|^2d\xi\\
&=
\int_0^{2\pi}d\varphi_{n-1}\int_0^{\pi}\sin\varphi_{n-2}d\varphi_{n-2}\cdot\cdot\cdot\int_0^{\pi}\sin\varphi_1^{n-2}d\varphi_{1}\int_0^{s}k^{n-1}
\big|\widehat{f}(k\hat{\xi})\big|^2dk.\\
\end{split}
\end{equation}
Combining (\ref{Fof}) and (\ref{I1s}), we get 
\begin{equation}\label{I12}
|I_1(s)|\leq C\epsilon^2 \ \ \mbox{for all} \ s\in[0,K],
\end{equation}
where $C>0$ depends on $R$ and $n$. Noting that the analyticity of $I_1(s)$ in (\ref{I1s}) only depends on $\int_0^{s}k^{n-1}
\big|\widehat{f}(k\hat{\xi})\big|^2dk$ with
$\widehat{f}$ the Fourier transform of $f$ given by
 $\widehat{f}(k\hat{\xi})=\int_{\R^n}f(x)e^{-ik\hat{\xi}\cdot x}dx$, we then substitute
$k=st$, $t\in(0,1)$  to obtain 
\begin{equation}\label{right}
\begin{split}
\int_0^{s}k^{n-1}
\big|\widehat{f}(k\hat{\xi})\big|^2dk&=\int_0^{1}s^nt^{n-1}
|\int_{\R^n}f(x)e^{-ist\hat{\xi}\cdot x}dx|^2dt\\
&=\int_0^{1}s^nt^{n-1}
(\int_{\R^n}f(x)e^{-ist\hat{\xi}\cdot x}dx)(\int_{\R^n}\overline{f(x)}e^{ist\hat{\xi}\cdot x}dx)dt.
\end{split}
\end{equation}
Clearly, the right-hand side of (\ref{right}) is an analytic function of $s=s_1+is_2\in\C$, $s_1,s_2\in\R$. 
Thus $I_1(s)$ is an entire analytic function of $s\in\C$ and the following elementary estimates hold.

\begin{lemma}\label{I1}
Let $||f||_{L^2(\R^n)}\leq M$.
Then we have for all $s=s_1+is_2\in\C$ that
\begin{equation}\label{LI}
|I_1(s)|\leq Cs^ne^{2R|s_2|}M^2,
\end{equation}
where $C>0$ depends on $n$ and $R$.
\end{lemma}
\begin{proof}
Let $k=st$, $t\in(0,1)$. A simple calculations yields 
\begin{equation}
\nonumber
\begin{split}
I_1(s)&=\int_{|\xi|\leq s}|\widehat{f}(\xi)|^2d\xi\\
&=
\int_0^{2\pi}d\varphi_{n-1}\int_0^{\pi}\sin\varphi_{n-2}d\varphi_{n-2}\cdot\cdot\cdot\int_0^{\pi}\sin\varphi_1^{n-2}d\varphi_{1}\int_0^{s}k^{n-1}
|\widehat{f}(k\hat{\xi})|^2dk\\
&=\int_0^{2\pi}d\varphi_{n-1}\int_0^{\pi}\sin\varphi_{n-2}d\varphi_{n-2}\cdot\cdot\cdot\int_0^{\pi}\sin\varphi_1^{n-2}d\varphi_{1}\int_0^{1}s^nt^{n-1}
|\widehat{f}(st\hat{\xi})|^2dt\\
&=\int_0^{2\pi}d\varphi_{n-1}\int_0^{\pi}\sin\varphi_{n-2}d\varphi_{n-2}\cdot\cdot\cdot\int_0^{\pi}\sin\varphi_1^{n-2}d\varphi_{1}\int_0^{1}s^nt^{n-1}
|\int_{\R^n}f(x)e^{-ist\hat{\xi}\cdot x}dx|^2dt.
\end{split}
\end{equation}
Noting that $|e^{-ist\hat{\xi}\cdot x}|\leq e^{R|s_2|}$ for all $x\in B_R$, we obtain by using the Schwarz inequality that 
\begin{equation}
\nonumber
|I_1(s)|\leq Cs^ne^{2R|s_2|}\int_{\R^n}|f(x)|^2dx,
\end{equation}
where $C>0$ depends on $n$ and $R$,
Then we complete the proof of (\ref{LI}).
\end{proof}
Denote
\begin{equation}
\nonumber
\begin{split}
I_2(s)&=\int_{|\xi|> s}|\widehat{f}(\xi)|^2d\xi\\
&=
\int_0^{2\pi}d\varphi_{n-1}\int_0^{\pi}\sin\varphi_{n-2}d\varphi_{n-2}\cdot\cdot\cdot\int_0^{\pi}\sin\varphi_1^{n-2}d\varphi_{1}\int_s^{+\infty}k^{n-1}
\big|\widehat{f}(k\hat{\xi})\big|^2dk.\\
\end{split}
\end{equation}
Now we estimate $I_2(s)$.
\begin{lemma}\label{lem3.8}
Let $f\in \mathcal{C}_{M,d}$.
For any $s>0$, we have 
\begin{equation}\label{I_2}
|I_2(s)|\leq C\frac{M^2}{s^{4d-n}},
\end{equation}
where $C>0$ depends on $n$ and $R$.
\end{lemma}
\begin{proof} 
Let $\Delta f(x)=\partial_{x_1}^2f(x)+\partial_{x_2}^2f(x)+\cdot\cdot\cdot+\partial_{x_n}^2f(x)$.
Denote $\Delta^mf(x):=\underbrace{\Delta\cdot\cdot\cdot\Delta(\Delta }_{m}f(x))$.
Since
$\widehat{f}$ is the Fourier transform of $f$ given by 
\begin{equation}\label{Ff}
\widehat{f}(k\hat{\xi})=\int_{\R^3}e^{-ik\hat{\xi}\cdot x}\; f(x)dx.
\end{equation}
Multiplying $(-ik\hat{\xi}_j)^{2}$,\ $j=1, 2,\cdot\cdot\cdot,n$ on both sides of (\ref{Ff}), we obtain 
\begin{eqnarray}
(-ik\hat{\xi}_j)^{2}\widehat{f}(k\hat{\xi})&=&(-ik\hat{\xi}_j)^{2}\int_{\R^3}e^{-ik\hat{\xi}\cdot x}\; f(x)dx\cr
&=&\int_{\R^3}\partial_{x_j}^2(e^{-ik\hat{\xi}\cdot x})\; f(x)dx\cr
&=&\int_{\R^3}(e^{-ik\hat{\xi}\cdot x})\; \partial_{x_j}^2f(x)dx.
\end{eqnarray}
Adding the above equality for $j=1, 2,\cdot\cdot\cdot,n$ with the aid of $\hat{\xi}_1^2+\hat{\xi}_2^2+\cdot\cdot\cdot+\hat{\xi}_n^2=1$ now gives 
\begin{equation}
(-ik)^{2}\widehat{f}(k\hat{\xi})=\int_{\R^3}(e^{-ik\hat{\xi}\cdot x})\;\Delta f(x)dx.
\end{equation}
Similarly, repeating the above process $d$ times shows that
\begin{equation}\label{nf}
(-ik)^{2d}\widehat{f}(k\hat{\xi})=\int_{\R^3}(e^{-ik\hat{\xi}\cdot x})\;\Delta^d f(x)dx.
\end{equation} 
Hence one immediately has from (\ref{nf}) that 
\begin{equation}\nonumber
|\widehat{f}(k\hat{\xi})|\leq C\frac{M^2}{k^{2d}}\quad\mbox{for\ all}\quad \hat{\xi}\in \Sp^{n-1},   \ k>0,
\end{equation}
where $C>0$ depends on $n$ and $R$. This leads to
\begin{equation}
\nonumber
\begin{split}
I_2(s)&=
\int_0^{2\pi}d\varphi_{n-1}\int_0^{\pi}\sin\varphi_{n-2}d\varphi_{n-2}\cdot\cdot\cdot\int_0^{\pi}\sin\varphi_1^{n-2}d\varphi_{1}\int_s^{+\infty}k^{n-1}
|\widehat{f}(k\hat{\xi})|^2dk\\
&\leq C\frac{M^2}{s^{4d-n}},
\end{split}
\end{equation}
which implies (\ref{I_2}).
\end{proof}
Let us recall the following result, which is proved in \cite{[7]}.
\begin{lemma}\label{2}
  Let $J(z)$ be an analytic function in $S=\{z=x+iy\in \mathbb{C}:-\frac{\pi}{4}<\arg z<\frac{\pi}{4}\}$ and continuous in $\overline{S}$ satisfying
  \begin{equation}
  \nonumber
  \begin{cases}
    |J(z)| \leq \epsilon, \  & z\in (0,L], \\
    |J(z)| \leq V, \  & z\in S,\\
    |J(0)|  =0.
  \end{cases}
  \end{equation}
Then there exists a function $\mu(z)$ satisfying
  \begin{equation}
  \nonumber
  \begin{cases}
   \mu (z)  \geq\frac{1}{2},\ \ &z\in (L,2^{\frac{1}{4}}L), \\
   \mu (z)  \geq\frac{1}{\pi}((\frac{z}{L})^4-1)^{-\frac{1}{2}},\ \ & z\in (2^{\frac{1}{4}}L, +\infty)
  \end{cases}
  \end{equation}
   such that
\begin{equation}
\nonumber
 |J(z)|\leq V\epsilon^{\mu(z)} \quad \mbox{for all} \quad z\in (L, +\infty).
\end{equation}
\end{lemma}
Using Lemma \ref{2}, we show the relation between $I_1(s)$ for $s\in(K,\infty)$ with $I_1(K)$.
\begin{lemma}
Let $||f||_{L^2(\R^n)}\leq M$. Then there exists a function $\mu(s)$ satisfying 
  \begin{equation}
  \nonumber
  \begin{cases}
   \mu (s)  \geq\frac{1}{2},\ \ &s\in (K,2^{\frac{1}{4}}K), \\
   \mu (s)  \geq\frac{1}{\pi}((\frac{s}{K})^4-1)^{-\frac{1}{2}},\ \ & s\in (2^{\frac{1}{4}}K, +\infty)
  \end{cases}
  \end{equation}
such that 
\begin{equation}\label{CM}
|I_1(s)|\leq CM^2e^{(2R+1)s}\epsilon^{2\mu(s)} \quad \mbox{for all}
\ K<s<+\infty,
\end{equation}
where $C>0$ depends on $n$ and $R$.
\end{lemma}
\begin{proof}
Let the sector  $S\subset\C$ be given in Lemma \ref{2}. Observe that $|s_2|\leq s_1$ when $s\in S$. It follows from Lemma \ref{I1} that 
\[|I_1(s)e^{-(2R+1)s}|\leq CM^2,\]
where $C>0$ depends on $n$ and $R$.
Recalling from (\ref{I12}) a prior estimate
$
|I_1(s)|\leq C\epsilon^2, \ s\in[0, K]
$.
Then applying Lemma \ref{2} with $L=K$ to be function $J(s):=I_1(s)e^{-(2R+1)s}$,
we conclude that  there exists a function $\mu(s)$ satisfying 
\begin{equation}
\nonumber
\begin{cases}
\mu (s)  \geq\frac{1}{2},\ \ &s\in (K, 2^{\frac{1}{4}}K), \\
\mu (s)  \geq\frac{1}{\pi}((\frac{s}{K})^4-1)^{-\frac{1}{2}},\ \ & s\in (2^{\frac{1}{4}}K, \infty)
\end{cases}
\end{equation}
such that
\begin{equation}\nonumber
|I_1(s)e^{-(2R+1)s}|\leq CM^2\epsilon^{2\mu(s)},
\end{equation}
where $K<s<+\infty$ and $C$ depending on $n$ and $R$. Thus we complete the proof.
\end{proof}
Now we show the proof of Theorem \ref{1}.
If $\epsilon\geq e^{-1}$, then the estimate is 
obvious.
If $\epsilon<e^{-1}$, we discuss (\ref{f}) in two cases.

Case (i):  $2^{\frac{1}{4}}((2R+3)\pi)^{\frac{1}{3}} K^{\frac{1}{3}}< |\ln \epsilon|^{\frac{1}{4}}$. 
Choose $s_0=\frac{1}{((2R+3)\pi)^{\frac{1}{3}}}K^{\frac{2}{3}}|\ln \epsilon|^{\frac{1}{4}}$. It is easy to get $s_0>2^{\frac{1}{4}}K$, then 
\[-\mu(s_0)\leq-\frac{1}{\pi}((\frac{s_0}{K})^4-1)^{-\frac{1}{2}}\leq-\frac{1}{\pi}(\frac{K}{s_0})^2.\]
A direct application of  estimate (\ref{CM}) shows that 
 \begin{eqnarray*}
 |I_1(s_0)|&\leq& CM^2\epsilon^{2\mu(s_0)}e^{(2R+3)s_0}\\
        &\leq& CM^2 e^{(2R+3)s_0-2\mu(s_0)|\ln\epsilon|}\\
           &\leq& CM^2 e^{(2R+3)s_0-\frac{2|\ln \epsilon|}{\pi}(\frac{K}{s_0})^2 }\\
           &=& CM^2 e^{-2(\frac{(2R+3)^2}{\pi})^{\frac{1}{3}}K^{\frac{2}{3}}|\ln \epsilon|^{\frac{1}{2}}(1- \frac{1}{2}|\ln \epsilon|^{-\frac{1}{4}})}.
 \end{eqnarray*}
Noting that $1-\frac{1}{2}|\ln \epsilon|^{-\frac{1}{4}}>\frac{1}{2}$ and $(\frac{(2R+3)^2}{\pi})^{\frac{1}{3}}>1$, we have 
\[|I_1(s_0)|\leq CM^2 e^{-K^{\frac{2}{3}}|\ln \epsilon|^{\frac{1}{2}}}.\]
Using the inequality $e^{-t}\leq\frac{(12d-3n)!}{t^{3(4d-n)}}$ for $t>0$, we get
\begin{equation}\label{ess}
  |I_1(s_0)|\leq C\frac{M^2 }{(K^2|\ln \epsilon|^{\frac{3}{2}})^{4d-n}}.
\end{equation}
Hence
\begin{equation}\label{ess1}
  \begin{split}
    (2\pi)^n||f||^2_{L^2{(\R^n)}}=||\widehat{f}||^2_{L^2{(\R^n)}}&= I_1(s_0)+I_2(s_0)\\
    &\leq C\frac{M^2 }{(K^2|\ln \epsilon|^{\frac{3}{2}})^{4d-n}}+C\frac{M^2 }{(K^{\frac{2}{3}}|\ln \epsilon|^{\frac{1}{4}})^{4d-n}}\\
    &\leq C\frac{M^2 }{(K^{\frac{2}{3}}|\ln \epsilon|^{\frac{1}{4}})^{4d-n}}.
  \end{split}
\end{equation}
Since $K^2|\ln \epsilon|^{\frac{3}{2}} \geq K^{\frac{2}{3}}|\ln \epsilon|^{\frac{1}{4}}$ when $K>1$ and $|\ln\epsilon|\geq1$.

Case (ii):  $|\ln \epsilon|^{\frac{1}{4}}\leq 2^{\frac{1}{4}}((2R+3)\pi)^{\frac{1}{3}} K^{\frac{1}{3}}$. In this case we choose $s_0=K$, then $s_0\geq 2^{-\frac{1}{4}}((2R+3)\pi)^{-\frac{1}{3}}K^{\frac{2}{3}}|\ln \epsilon|^{\frac{1}{4}}$. 
Using  estimate (\ref{I12}), we obtain
\begin{equation}\label{ess2}
  \begin{split}
    (2\pi)^n||f||^2_{L^2{(\R^n)}}=||\widehat{f}||^2_{L^2{(\R^n)}}&= I_1(s_0)+I_2(s_0)\\
&\leq C\big(\epsilon^2+\frac{ M^2}{(K^{\frac{2}{3}}|\ln \epsilon|^{\frac{1}{4}})^{4d-n}}\big).
\end{split}
\end{equation}
Combining (\ref{ess1}) and (\ref{ess2}), we finally get
\begin{equation}\nonumber
 ||f||^2_{L^2{(\R^n)}}\leq C\big(\epsilon^2+\frac{ M^2}{(K^{\frac{2}{3}}|\ln \epsilon|^{\frac{1}{4}})^{4d-n}}\big).
\end{equation}

\section*{Acknowledgment}
The work of Suliang Si is supported by  the Shandong Provincial Natural Science Foundation (No. ZR202111240173). The author would like to thank Guanghui Hu and Jiaqing Yang for helpful discussions. This work does not have any conflicts of interest.

\end{document}